\documentclass[a4paper, reqno]{amsart}

\usepackage{amsmath}
\usepackage{amsfonts}
\usepackage{amsthm}
\usepackage{graphicx}
\usepackage{eucal}
\usepackage{amscd}
\usepackage[all,2cell]{xy}
\usepackage{amssymb}
\usepackage{mathrsfs}

 \usepackage{tikz}
 \usetikzlibrary{positioning}
 \tikzset{mynode/.style={draw,circle,inner sep=1pt,outer sep=0pt}}

\newtheorem{teo}{Theorem}[section]
\newtheorem{cor}[teo]{Corollary}
\newtheorem{lem}[teo]{Lemma}
\newtheorem{defi}[teo]{Definition}

\newtheorem{prop}[teo]{Proposition}

\newcommand{\X}{\ensuremath{\mathbb{X}}}
\newcommand{\A}{\ensuremath{\mathbb{A}}}

\newdir{ |>}{{}*!/-3.5pt/@{|}*!/-8pt/:(1,-.2)@^{>}*!/-8pt/:(1,+.2)@_{>}}

\dedicatory{}

\begin{document}

\title[CENTRAL EXTENSIONS AND WEAKLY ACTION REPRESENTABLE CATEGORIES]{CENTRAL EXTENSIONS OF ASSOCIATIVE ALGEBRAS AND WEAKLY ACTION REPRESENTABLE CATEGORIES}

\author{George Janelidze}
\address[George Janelidze]{Department of Mathematics and Applied Mathematics, University of Cape Town, Rondebosch 7700, South Africa}
\thanks{}
\email{george.janelidze@uct.ac.za}


\makeatletter\@namedef{subjclassname@2020}{\textup{2020} Mathematics Subject Classification}\makeatother


\begin{abstract}
A central extension is a regular epimorphism in a Barr exact category $\mathscr{C}$ satisfying suitable conditions involving a given Birkhoff subcategory of $\mathscr{C}$ (joint work with G. M. Kelly, 1994). In this paper we take $\mathscr{C}$ to be the category of (not-necessarily-unital) algebras over a (unital) commutative ring and consider central extensions with respect to the category of commutative algebras. We propose a new approach that avoids the intermediate notion of central extension due to A. Fr\"ohlich in showing that $\alpha:A\to B$ is a central extension if and only if $aa'=a'a$ for all $a,a'\in A$ with $\alpha(a')=0$. This approach motivates introducing what we call \textit{weakly action representable categories}, and we show that such categories are always action accessible. We also make remarks on what we call \textit{initial weak representations of actions} and formulate several open questions.   
\end{abstract}

\date{\today}

\maketitle

\tableofcontents

\section{Introduction}

Let $K$ be a fixed commutative ring (with $1$). A $K$\textit{-algebra} is a $K$-module $A$ equipped with an associative multiplication with $k(aa')=(ka)a'=a(ka')$ for all $k\in K$ and $a,a'\in A$. Let $\mathrm{Comm}:\mathsf{Alg}(K)\to\mathsf{CommAlg}(K)$ be the reflection of the category $\mathsf{Alg}(K)$ of $K$-algebras to its full subcategory $\mathsf{CommAlg}(K)$ of commutative $K$-algebras. Let $R:\mathsf{Alg}(K)\to\mathsf{Alg}(K)$ be the functor defined by
$$R(A)=\mathrm{Ker}(\eta_A:A\to\mathrm{Comm}(A))=\langle\{aa'-a'a\mid a,a'\in A\}\rangle,$$
in obvious notation. 

The following definition introduces three notions that are in fact equivalent to each other:

\begin{defi}
	A regular epimorphism (=surjective homomorphism) $\alpha:A\to B$ of $K$-algebras is said to be
	\begin{itemize}
		\item [(a)] classically-central, if $aa'=a'a$ for all $a,a'\in A$ with $\alpha(a')=0$;
		\item [(b)] algebraically-central, if for every two parallel morphisms $u$ and $v$ in $\mathsf{Alg}(K)$ with codomain $A$, we have
		$$\alpha u=\alpha v\Rightarrow R(u)=R(v);$$
		\item [(c)] categorically-central, if the diagram
		$$\xymatrix{A\times_BA\ar[d]_{\pi}\ar[rr]^-{\eta_{A\times_BA}}&&\mathrm{Comm}(A\times_BA)\ar[d]^{\mathrm{Comm}(\pi)}\\A\ar[rr]_-{\eta_A}&&\mathrm{Comm}(A)}$$
		where $\pi$ is any of the two pullback projections $A\times_BA\to A$ and the horizontal arrows are unit components of the reflection $\mathrm{Comm}$, is a pullback. 
	\end{itemize}
\end{defi}

In this definition:
\begin{itemize}
	\item (b) copies a modification, due to A. S.-T. Lue \cite{[L1967]}, of A. Fr\"ohlich's definition of a central extension given in \cite{[F1963]}, except that both of these authors considered a more general case where the role of $\mathsf{CommAlg}(K)$ is played by an arbitrary subvariety of $\mathsf{Alg}(K)$; see also J. Furtado-Coelho \cite{[F-C1976]}, T. Everaert \cite{[E2007]}, and T. Everaert and T. Van der Linden \cite{[EV2012]} (and references therein) for more general contexts. The term algebraically-central was used for a different notion in \cite{[JK2000a]}.
	\item The equivalence of ``classically-central'' and ``algebraically-central'' would probably be considered obvious by Fr\"ohlich and Lue, and it is easy indeed. However, it is not even mentioned, neither in \cite{[F1963]} nor in \cite{[L1967]}.
	\item To say that $\alpha:A\to B$ is categorically-central is the same as to say that it is central in $\mathsf{Alg}(K)$ with respect to $\mathsf{CommAlg}(K)$, or, equivalently, normal in $\mathsf{Alg}(K)$ with respect to $\mathsf{CommAlg}(K)$ in the sense of \cite{[JK1994]} (see Subsection 1.3 therein and the results it refers to). The equivalence of ``algebraically-central'' and ``categorically-central'' is a special case of a part of Theorem 5.2 in \cite{[JK1994]}.   
\end{itemize}
The purpose of this paper is three-fold:
\begin{itemize}
	\item [(i)] To prove the equivalence of ``classically-central'' and ``categorically-central'' directly, not using the intermediate notion of algebraically-central extension (Section 3). We will use instead what one might call the theory of split extensions of $K$-algebras (presented in Section 2); it is quite simple and certainly known for a very long time, although it does not seem to be ever presented explicitly.
	\item [(ii)] To introduce a new notion of \textit{weakly action representable} (semi-abelian) category (Section 4) motivated by that proof, to show that every weakly action representable category is action accessible in the sense of \cite{[BJ2009]}, and to make remarks (Section 5) on what we call \textit{initial weak representations of actions}, which exists, for instance, in every total weakly action representable category.
	\item [(iii)] To formulate related open questions (Section 6).
\end{itemize}
Sections 2 (most of which rephrases a small, and in fact known before, part of Section 2 of \cite{[BJK2005a]}) and 3 are self-contained, while Section 4 requires familiarity with \textit{protomodular} \cite{[B1991]}, \textit{semi-abelian} \cite{[JMT2002]}, \textit{action representable} \cite{[BJK2005]} (where a longer name ``has representable object actions'' was used), and \textit{action accessible} \cite{[BJ2009]} categories.

\section{Split extensions of algebras}

An \textit{action} (of $K$-algebras) is a system $(B,X,l,r)$ in which $B$ and $X$ are $K$-algebras and $l:B\to\mathrm{End}_K(X)$ and $r:B\to\mathrm{End}_K(X)^{\mathrm{op}}$ are $K$-algebra homomorphisms written as
$$l(b)(x)=bx\,\,\,\text{and}\,\,\,r(b)(x)=xb$$
and satisfying the four equalities:
$$b(xb')=(bx)b',\,\,\,b(xx')=(bx)x',\,\,\,x(bx')=(xb)x',\,\,\,x(x'b)=(xx')b$$
for all $b,b'\in B$ and $x,x'\in X$. Here $\mathrm{End}_K(X)$ denotes the $K$-algebra of $K$-module endomorphisms of $X$. Note that since we required $l$ and $r$ to be homomorphisms we also have
$$b(b'x)=(bb')x,\,\,\,x(bb')=(xb)b',$$
and so we simply have \textit{mixed associativity}
$$u(vw)=(uv)w$$
for all $u,v,w\in B\cup X$ (assuming $B$ and $X$ to be disjoint).  

A \textit{split extension} (of $K$-algebras) is a diagram $(A,B,X,\alpha,\beta,\kappa)=$
$$\xymatrix{X\ar[r]^-\kappa&A\ar@<0.7ex>[r]^-\alpha&B\ar[l]^-\beta}$$ 
in $\mathsf{Alg}(K)$, in which $\alpha\beta=1_B$ and $(X,\kappa)$ is a kernel of $\alpha$.

Given an action $(B,X,l,r)$, the \textit{associated split epimorphism} is defined as the diagram $(B\times^{(l,r)}X,B,X,\pi_1,\iota_1,\iota_2)=$
$$\xymatrix{X\ar[r]^-{\iota_2}&B\times^{(l,r)}X\ar@<0.7ex>[r]^-{\pi_1}&B\ar[l]^-{\iota_1},}$$
in which:
\begin{itemize}
	\item $B\times^{(l,r)}X=B\times X$ as $K$-modules;
	\item the multiplication on $B\times^{(l,r)}X$ is defined by 
	$$(b,x)(b',x')=(bb',bx'+xb'+xx');$$
	\item $\pi_1$, $\iota_1$, and $\iota_2$ are defined by $\pi_1(b,x)=b$, $\iota_1(b)=(b,0)$, and $\iota_2(x)=(0,x)$, respectively.
\end{itemize}

Conversely, given a split extension $(A,B,X,\alpha,\beta,\kappa)$, the associated action is defined as $(B,X,l,r)$, where $l$ and $r$ are uniquely determined the equalities 
$$\kappa(l(b)(x))=b\kappa(x)\,\,\,\text{and}\,\,\,\kappa(r(b)(x))=\kappa(x)b,$$ 
respectively.
This determines a span equivalence
$$\xymatrix{&\mathsf{Act}_K\ar[dl]_{(B,X,l,r)\mapsto B}\ar@<-0.4ex>[dd]\ar[dr]^{(B,X,l,r)\mapsto X}\\\mathsf{Alg}_K&&\mathsf{Alg}_K\\&\mathsf{SplExt}_K\ar[ul]^{(A,B,X,\alpha,\beta,\kappa)\mapsto B\,\,\,\,\,}\ar@<-0.4ex>[uu]\ar[ur]_{\,\,\,\,\,(A,B,X,\alpha,\beta,\kappa)\mapsto X}}$$
between the spans of actions and of split extensions. Here a morphism  $$(f,g,h):(A,B,X,\alpha,\beta,\kappa)\to(A',B',X',\alpha',\beta',\kappa')$$
of spans is a triple $(f,g,h)$, or, more formally a diagram 
$$\xymatrix{X\ar[d]_h\ar[r]^-\kappa&A\ar[d]_f\ar@<0.4ex>[r]^-\alpha&B\ar[d]^g\ar@<0.4ex>[l]^-\beta\\X'\ar[r]_-{\kappa'}&A'\ar@<0.4ex>[r]^-{\alpha'}&B'\ar@<0.4ex>[l]^-{\beta'}}$$
in which $\kappa'h=f\kappa$, $\alpha'f=g\alpha$, and $f\beta=\beta'g$, while a morphism 
$$(g,h):(B,X,l,r)\to(B',X',l',r')$$
has $g:B\to B'$ and $h:X\to X'$ satisfying $h(bx)=g(b)h(x)$ and $h(xb)=h(x)g(b)$. And, under the functors represented by the vertical arrow in the diagram above, $(f,g,h)\in\mathsf{SplExt}_K$ corresponds to $(g,h)\in\mathsf{Act}_K$. Nore also that, in the triple $(f,g,h)$, $g$ anf $h$ determine $f$ by 
$$f(\kappa(x)+\beta(b))=\kappa'h(x)+\beta'g(b),$$
having in mind that every $a\in A$ can be (uniquely) presented as $a=\kappa(x)+\beta(b)$ with $\kappa(x)=b-\beta\alpha(a)$ and $b=\alpha(a)$.

\begin{prop}
	Let $(B,X,l,r)$ be an action, $g:B\to B'$ a surjective $K$-algebra homomorphism, and $l':B'\to\mathrm{End}_K(X)$ and $r':B'\to\mathrm{End}_K(X)^{\mathrm{op}}$ maps (of sets) with $l=l'g$ and $r=r'g$. Then $(B',X,l',r')$ be an action, and (obviously) 
	$$(g,1_X):(B,X,l,r)\to(B',X,l',r')$$
	is a morphism of actions. 
\end{prop}
\begin{proof}
	First note that, since $g$ is a surjective $K$-algebra homomorphism and $l$ and $r$ are $K$-algebra homomorphisms, the equalities $l=l'g$ and $r=r'g$ imply that $l'$ and $r'$ are $K$-algebra homomorphisms.
	After that, since $g$ is surjective, it suffices to verify the equalities
	$$g(b)(xg(b'))=(g(b)x)g(b'),\,\,\,g(b)(xx')=(g(b)x)x'$$
	$$x(g(b)x')=(xg(b))x',\,\,\,x(x'g(b))=(xx')g(b)$$
	for all $b,b'\in B$ and $x,x'\in X$; here $g(b)x=l'(g(b))(x)=l(b)(x)=bx$, etc. However, these four equalities immediately follow from the four equalities given at the beginning of this section.
\end{proof} 

\section{Symmetric actions and central extensions}

Let us call an action $(B,X,l,r)$ \textit{symmetric}, if $l=r$. 

\begin{lem}
	Let $(A,B,X,\alpha,\beta,\kappa)$ and $(B,X,l,r)$ be a split extension and an action corresponding to each other. Then the following conditions are equivalent:
	\begin{itemize}
		\item [(a)] $\alpha$ is classically-central;
		\item [(b)] $a\kappa(x)=\kappa(x)a$ for all $a\in A$ and $x\in X$;
		\item [(c)] $(b,x)(0,x')=(0,x')(b,x)$ in $B\times^{(l,r)}X$ for all $b\in B$ and $x,x'\in X$;
		\item [(d)] $X$ is commutative and $(B,X,l,r)$ is symmetric;
		\item [(e)] condition (d) holds and $l:B\to\mathrm{End}_K(X)$ factors through $\eta_B:B\to I(B)$;
		\item [(f)] $X$ is commutative and there exist a symmetric action of the form $(I(B),X,\bar{l},\bar{r})$ such that
		            $$(\eta_B,1_X):(B,X,l,r)\to(I(B),X,\bar{l},\bar{r})$$
		            is a morphism in $\mathsf{Act}_K$;
		\item [(g)] condition (f) holds with $I(B)\times^{(\bar{l},\bar{r})}X$ being commutative.
	\end{itemize}
\end{lem}
\begin{proof}
	(a)$\Leftrightarrow$(b) is obvious.
	
	Without loss of generality we can assume 
	$$(A,B,X,\alpha,\beta,\kappa)=(B\times^{(l,r)}X,B,X,\pi_1,\iota_1,\iota_2),$$
	which makes (b)$\Leftrightarrow$(c) obvious too.
	
	(c)$\Leftrightarrow$(d) follows from
	$$(b,x)(0,x')=(0,x')(b,x)\Leftrightarrow bx'+xx'=x'b+x'x,$$
	since it holds for all $b\in B$ (including the case $b=0$) and $x,x'\in X$.
	
	(d)$\Leftrightarrow$(e): Since $l=r$ is a $K$-algebra homomorphism, it suffices to probe that $l(bb')=l(b'b)$ for all $b,b'\in B$. We have
	$$l(bb')(x)=(bb')x=b(b'x)=(b'x)b=b'(xb)=b'(bx)=(b'b)x=l(b'b)(x),$$
	using the mixed associativity and the symmetry.
	
	(e)$\Rightarrow$(f) follows from Proposition 2.1, while (f)$\Rightarrow$(e) is obvious.
	
	Now it remains to show that conditions (a)-(f) imply the commutativity of $I(B)\times^{(\bar{l},\bar{r})}X$, but this commutativity easily follows from (c) and the commutativity of $X$ and of $I(B)$. 
\end{proof}

\begin{lem}
	Let
	$$\xymatrix{A\ar[d]_f\ar[r]^\alpha&B\ar[d]^g\\A'\ar[r]_{\alpha'}&B'}$$
	be a pullback diagram in $\mathsf{Alg}(K)$. Then 
	\begin{itemize}
		\item [(a)] If $\alpha'$ is classically-central, then $\alpha$ is classically-central;
		\item [(b)] If $\alpha$ is classically-central and $g$ is surjective, then $\alpha'$ is classically-central.
	\end{itemize}
\end{lem}
\begin{proof}
	Just use straightforward calculation.
\end{proof}

\begin{teo}
	For a split epimorphism $\alpha:A\to B$ of $K$-algebras, the following conditions are equivalent:
	\begin{itemize}
		\item [(a)] $\alpha$ is classically-central;
		\item [(b)] condition (a) holds and there is a morphism of split extensions of the form
		$$\xymatrix{X\ar@{=}[d]\ar[r]&A\ar[d]_f\ar@<0.4ex>[r]^-\alpha&B\ar[d]^{\eta_B}\ar@<0.4ex>[l]\\X\ar[r]&A'\ar@<0.4ex>[r]^-{\alpha'}&I(B)\ar@<0.4ex>[l]}$$
		which, once the top row $(A,B,X,\alpha,\beta,\kappa)$ is chosen for the given $\alpha$, corresponds to the morphism of actions displayed in Lemma 3.1(f);
		\item [(c)] the diagram
		$$\xymatrix{A\ar[d]_{\eta_A}\ar[rr]^\alpha&&B\ar[d]^{\eta_B}\\\mathrm{Comm}(A)\ar[rr]_{\mathrm{Comm}(\alpha)}&&\mathrm{Comm}(B)}$$
		is a pullback;
		\item [(d)] there exists a pullback diagram of the form
		$$\xymatrix{A\ar[d]_f\ar[r]^\alpha&B\ar[d]^g\\A'\ar[r]_{\alpha'}&B'}$$
		with commutative $A'$ and $B'$;
		\item [(e)] conditions (d) holds with all the homomorphisms involved being surjective.    
	\end{itemize}
\end{teo}
\begin{proof}
	(a)$\Leftrightarrow$(b) follows from Lemma 3.1.
	
	(b)$\Rightarrow$(e) follows from the fact that the square in (b) formed by $\alpha$, $f$, $\eta_B$, and $\alpha'$ is a pullback (since $\alpha$ and $\alpha'$ have canonically isomorphic kernels), and $A'$ in (b) is commutative by Lemma 3.1(a)$\Rightarrow$(g) and (a)$\Rightarrow$(b).
	
    (c)$\Leftrightarrow$(e) follows from Theorem 3.4(d) and (4.2) of \cite{[JK1994]}.
    
    (c)$\Rightarrow$(d) is trivial and (d)$\Rightarrow$(a) follows from Lemma 3.2(a). 	
\end{proof}

\begin{teo}
	For a reqular epimorphism (=surjective homomorphism) $\alpha:A\to B$ of $K$-algebras, the following conditions are equivalent:
	\begin{itemize}
		\item [(a)] $\alpha$ is classically-central;
		\item [(b)] $\alpha$ is categorically-central.
	\end{itemize}
\end{teo}
\begin{proof}
	Just note that, for any $K$-algebra homomorphism $A\to B$, any of the pullback projections $A\times_BA\to A$ is a split epimorphism, and apply Lemma 3.2 and Theorem 3.3.
\end{proof}
The readers should forgive me for the display of diagram in Theorem 3.3(c) being rotated in comparison with the display of diagram in Definition 1.1(c): it is done so due to the tradition is displaying split extensions coming from homological algebra.

In ring theory, the centre of a $K$-algebra $A$ is always defined as the $K$-subalgebra $\{c\in A\mid\forall_{a\in A}\,ac=ca\}$ of $A$. Contrary to that, Theorem 3.4 suggests to define the centre of $A$ as the largest ideal $X$ of $A$ with $ax=xa$ for all $a\in A$ and $x\in X$. This brings us closer to commutator theory in universal algebra, according to which the centre of $A$ should be defined as the largest ideal $X$ of $A$ with $ax=0=xa$ for all $a\in A$ and $x\in X$.  	

\section{Weakly action representable categories} 

In this section we will consider split extensions in a fixed semi-abelian category $\mathscr{C}$, using, as far as possible, the same notation as for split extensions of $K$-algebras. Given an object $X$ in $\mathscr{C}$, we will consider the functor
$$\mathrm{SplExt}(-,X):\mathscr{C}^{\mathrm{op}}\to\mathsf{Sets},$$
where, for $X\in\mathscr{C}$, $\mathrm{SplExt}(B,X)$ is the set of isomorphism classes of split extensions $(A,B,X,\alpha,\beta,\kappa)$ with fixed $B$ and $X$, exactly as \cite{[BJK2005]}, except that following the style of notation of \cite{[BJK2005]} we would write $(B,A,\alpha,\beta,\kappa)$ instead of $(A,B,X,\alpha,\beta,\kappa)$. Recall that in this general context we also have an equivalence between split extensions and actions, and, in particilar, a canonical functor isomorphism $$\mathrm{SplExt}(-,X)\approx\mathrm{Act}(-,X)$$ (Theorem 6.2 of \cite{[BJK2005]}, which rephrases a special case of Theorem 3.4 of \cite{[BJ1998]}). Here an \textit{action} of $B$ on $X$, that is, an element of $\mathrm{Act}(B,X)$, is a triple $(B,X,\xi)$, in which 
$$\xymatrix{B\flat X\ar@{=}[d]\ar[r]^-\xi&X\\\mathrm{Ker}([1,0]:B+X\to B)}$$
satisfying suitable conditions (see Subsection 3.3 of \cite{[BJK2005]}). Let us also mention that in the case of $K$-algebras an action $(B,X,\xi)$ corresponds to the action $(B,X,l,r)$, in the sense of Section 2, defined by $l(b)(x)=\xi(b\otimes x)$ and $r(b)(x)=\xi(x\otimes b)$, using the description of $B\flat X$ in terms of tensor products.

\begin{defi}
	For an object $X$ in $\mathscr{C}$, a weak representation of actions on $X$ is a pair $(M,\mu)$ in which $M$ is an object in $\mathscr{C}$ and $$\mu:\mathrm{SplExt}(-,X)\to\mathrm{hom}(-,M)$$ is a monomorphism of functors. We will say that $\mathscr{C}$ is weakly action representable if every object in $\mathscr{C}$ has a weak representation of actions on it. 
\end{defi}

In the rest of this section $(M,\mu)$ will denote a fixed weak representation of actions on a given object $X$ in $\mathscr{C}$, unless stated otherwise. For a split extension $E=(A,B,X,\alpha,\beta,\kappa)$, the morphism $\mu_B([E]):B\to M$ will be called the \textit{acting morphism} corresponding to $E$. The following proposition immediately follows from our definitions and the fact that, in each such split extension, the morphisms $\beta$ and $\kappa$ are jointly epic:
 
\begin{prop}
	Let $E=(A,B,X,\alpha,\beta,\kappa)$ and $E'=(A',B',X,\alpha',\beta',\kappa')$ be split extensions. The following conditions on a morphism $g:B\to B'$ are equivalent:
	\begin{itemize}
		\item [(a)] there exists a morphism $f:A\to A'$ such that $(f,g,1_X):E\to E'$ is a morphism of split extensions;
		\item [(b)] there exists a unique morphism $f:A\to A'$ such that $(f,g,1_X):E\to E'$ is a morphism of split extensions;
		\item [(c)] $\mu_{B'}([E'])g=\mu_B([E])$.\qed
	\end{itemize}
\end{prop}

The following corollary generalizes Corollary 1.5 of \cite{[BJ2009]}:

\begin{cor}
	A split extension $E=(A,B,X,\alpha,\beta,\kappa)$ is faithful in the sense of \cite{[BJ2009]} if and only if its corresponding acting morphism $\mu_B([E]):B\to M$ is a monomorphism.\qed
\end{cor}

For a diagram of the form
$$\xymatrix{\bullet\ar[d]\ar[r]&\bullet\ar[d]\ar@<0.4ex>[r]\ar@<-0.4ex>[r]&\bullet\ar[d]\\\bullet\ar[r]&\bullet\ar@<0.4ex>[r]\ar@<-0.4ex>[r]&\bullet}$$
in any category, it is (well known and) easy to see that the left-hand square is a pullback whenever:
\begin{itemize}
	\item the top row is an equalizer diagram;
	\item the two composites in the bottom row are equal;
	\item the triple of vertical arrows is a diagram morphism from the top row to the bottom row;
	\item the left-hand bottom arrow and the right-hand vertical arrow are monomorphisms.
\end{itemize}
From this observation and the fact $\mathrm{SplExt}(-,X)$ transforms coequalizers of equivalence relations in $\mathscr{C}$ into equalizers in $\mathsf{Sets}$ (which is briefly proved in \cite{[BJK2005a]}), we easily obtain:
\begin{prop}
	If $g:B\to B'$ is a regular epimorphism in $\mathscr{C}$, then the diagram
	$$\xymatrix{\mathrm{SplExt}(B',X)\ar[d]_{\mathrm{SplExt}(g,X)}\ar[r]^{\mu_{B'}}&\mathrm{hom}(B',M)\ar[d]^{\mathrm{hom}(g,M)}\\\mathrm{SplExt}(B,X)\ar[r]_{\mu_{B}}&\mathrm{hom}(B,M)}$$
	is a pullback.\qed
\end{prop}

In terms of acting morphisms this proposition can be rephrased as:

\begin{prop}
	For a regular epimorphism $g:B\to B'$, a morphism $\varphi:B'\to M$ is an acting morphism if (and only if) so is $\varphi g$.\qed
\end{prop}

Recall that a semi-abelian category $\mathscr{C}$ is \textit{action accessible} in the sense of \cite{[BJ2009]} if, for every object $X$ in $\mathscr{C}$, every split extension $E=(A,B,X,\alpha,\beta,\kappa)$ admits a morphism of the form $(f,g,1_X):E\to E'$ with faithful $E'=(A',B',X,\alpha',\beta',\kappa')$. Since $\mathscr{C}$ admits (regular epi, mono) factorizations, from Corollary 4.3 and Proposition 4.4, we obtain:

\begin{teo}
	Every weakly action representable category is action accessible.\qed
\end{teo}  

The fact that the category of (not-necessarily-unital) rings is action accessible was proved in \cite{[BJ2009]}, and now Theorem 4.6 gives a new proof of it, more category-theoretic in a sense. Indeed, for rings, and, more generally, for $K$-algebras, given $X$ we can take $M$ above to be $\mathrm{End}_K(X)\times\mathrm{End}_K(X)^{\mathrm{op}}$ and define $\mu$ by $\mu_B([E])=\langle l,r\rangle$, where $(B,X,l,r)$ corresponds to $E=(A,B,X,\alpha,\beta,\kappa)$ as in Section 2 \textendash\, this shows that the category $\mathsf{Alg}(K)$ is weakly action representable, and then we can apply Theorem 4.6. Furthermore, Proposition 4.5 is in fact a category-theoretic counterpart of Proposition 2.1, which was my main motivation for introducing weakly action representable categories.

\section{Remarks on initial weak reprentations}

In this section, $\mathscr{C}$ denotes again a fixed semi-abelian category, and $X$ denotes an object in $\mathscr{C}$; dealing with split extensions in $\mathscr{C}$ we will shorten our notation: instead of $E=(A,B,X,\alpha,\beta,\kappa)$, let us write just $E_{B,X}$. 

Consider the commutative diagram 
$$\xymatrix{\mathsf{SplExt}(-,X)\ar[rr]^-{E_{B,X}\mapsto(B,[E_{B,X}])}\ar[dr]_{E_{B,X}\mapsto B\,\,\,\,\,}&&\mathrm{El}(\mathrm{SplExt}(-,X))\ar[dl]^{\,\,\,\,\,(B,[E_{B,X}])\mapsto B}\\&\mathscr{C}}$$
in which:
\begin{itemize}
	\item $\mathsf{SplExt}(-,X)$ is the category of all split extensions $E_{B,X}$ in $\mathscr{C}$ with fixed $X$ (and whose morphisms are identities on $X$);
	\item $\mathrm{El}(\mathrm{SplExt}(-,X))$ is the category of elements of the functor $\mathrm{SplExt}(-,X)$;
	\item the horizontal arrow is the `standard' equivalence and the two other arrows are the canonical functors defined as shown.
\end{itemize}
For the canonical functors above, let us write 
$$(L,\lambda)=\mathrm{colim}(\mathsf{SplExt}(-,X)\to\mathscr{C})=\mathrm{colim}(\mathrm{El}(\mathrm{SplExt}(-,X))\to\mathscr{C}),$$
where $L$ is the vertex of the colimiting cone, and, in terms of the first equality, $$\lambda=(\lambda_{E_{B,X}}:B\to L)_{E_{B,X}\in\mathsf{SplExt}(-,X)},$$
whenever this colimit exists.

Of course, identifying our two canonical functors with each other, $\lambda$ can also be seen as a natural transformation $$\mathrm{SplExt}(-,X)\to\mathrm{hom}(-,L)$$ with $\lambda_B([E_{B,X}])=\lambda_{E_{B,X}}$, and then $(L,\lambda)$ becomes nothing but a universal arrow $\mathrm{SplExt}(-,X)\to\mathrm{Y}$, where $$\mathrm{Y}:\mathscr{C}\to\mathsf{Sets}^{\mathscr{C}^{\mathrm{op}}}$$
is the Yoneda embedding of $\mathscr{C}$. Conversely, the existence of a universal arrow $(L,\lambda):\mathrm{SplExt}(-,X)\to\mathrm{Y}$ implies the existence of the colimit above with the same relationship between them, that is, with $\lambda_{E_{B,X}}$ defined by $\lambda_{E_{B,X}}=\lambda_B([E_{B,X}])$.

Furthermore, we immediately obtain:
\begin{prop}
	The following conditions are equivalent:
	\begin{itemize}
		\item [(a)] $X$ has representable object actions on it;
		\item [(b)] there exists a universal arrow $(L,\lambda):\mathrm{SplExt}(-,X)\to\mathrm{Y}$, for which\\ $\lambda:\mathrm{SplExt}(-,X)\to\mathrm{hom}(-,L)$ is an isomorphism.\qed 
	\end{itemize}
\end{prop}
\begin{prop}
	Let $(L,\lambda):\mathrm{SplExt}(-,X)\to\mathrm{Y}$ be a universal arrow. The following conditions are equivalent:
	\begin{itemize}
		\item [(a)] there exists a weak representation of actions on $X$;
		\item [(b)] $(L,\lambda)$ is a weak representation of actions on $X$.
		\item [(c)] $(L,\lambda)$ is an initial object in the category of weak representations of actions on $X$.\qed
	\end{itemize}
\end{prop}
In the situation of 5.1(c) we will simply say that $(L,\lambda)$ is an \textit{initial weak representations of actions on} $X$, and we have:
\begin{cor}
	If $\mathscr{C}$ is total in the sense of R. Street and R. Walters \cite{[SW1978]}, then the following conditions are equivalent:
	\begin{itemize}
		\item [(a)] there exists a weak representation of actions on $X$;
		\item [(b)] there exists an initial weak representation of actions on $X$.\qed 
	\end{itemize}
\end{cor}
We might call (a semi-abelian category) $\mathscr{C}$ \textit{weakly initially action representable} if every object in $\mathscr{C}$ has initial weak representations of actions on it. Then we have:
\begin{prop}
	Every action representable category is weakly initially action representable, every weakly initially action representable category is weakly action representable, and every total weakly action representable category is weakly initially action representable. In particular, every weakly action representable variety of universal algebras is weakly initially action representable.\qed 
\end{prop}

Note that additional `non-varietal' examples of total categories are due to W. Tholen \cite{[T1980]}, while semi-abelian monadic over $\mathsf{Sets}$ categories were considered by M. Gran and J. Rosick\'y \cite{[GR2004]} (for a different reason, but they are also total by a generalization \cite{[T1980]} of a result of \cite{[SW1978]}).  

\section{Open questions}

\textbf{6.1.} Theorem 3.4, which is the main result of Section 3, can be immediately deduced from Theorem 5.2 of \cite{[JK1994]} (stated in the context of varieties of groups with multiple operators in the sense of P. Higgins \cite{[H1956]}) and the fact that $\alpha:A\to B$ is classically-central if and only if it is algebraically central (see Definition 1.1 and the remarks below it). And there is a similar situation with the equivalence of three notions of centrality in the context considered in \cite{[JK2000a]}; in particular, this applies to $K$-algebras, where the commutator-theoretic counterpart of a classically-central extension is a surjective homomorphism $\alpha:A\to B$ with $aa'=0$ for all $a,a'\in A$ with $\alpha(a')=0$. We then observe:
\begin{itemize}
	\item To use the equalities like $aa'=0$ (just mentioned) or $aa'=a'a$ (in Definition 1.1(a)) is obviously simpler than to use the corresponding centrality conditions of A. Fr\"ohlich \cite{[F1963]} and A. S.-T. Lue \cite{[L1967]}.
	\item Therefore it would be interesting to find counterparts of classically-central extensions, if not in general, then at least for some other specific algebraic contexts that the theory of central extensions of \cite{[JK1994]} applies to.    
\end{itemize}

\textbf{6.2.} Is there any category-theoretic counterpart of what is done in Section 3 beyond the fact that Proposition 4.5 is a category-theoretic counterpart of Proposition 2.1, as mentioned at the end of Section 4? 

\textbf{6.3.} Is there any reasonable mild condition on a semi-abelian category under which its action accessibility implies that it is weakly action representable? I would not exclude a possibility that this implication holds for all semi-abelian varieties of universal algebras. This question is also interesting in connection with the results of A. Montoli \cite{[M2010]}. 

\textbf{6.4.} Motivated by the results of Section 5, it seems interesting to calculate the colimit $$(L,\lambda)=\mathrm{colim}(\mathsf{SplExt}(-,X)\to\mathscr{C})$$ considered in Section 5 in various concrete cases. For example, in the case $\mathscr{C}=\mathsf{Alg}(K)$ considered in Sections 2 and 3, it is easy to see that $L$ is not necessarily isomorphic to $\mathrm{End}_K(X)\times\mathrm{End}_K(X)^{\mathrm{op}}$; a better candidate for it (obviously) is
$$\{(\varphi,\psi)\in\mathrm{End}_K(X)\times\mathrm{End}_K(X)^{\mathrm{op}}\mid\forall_{x,x'\in X}\,x\varphi(x')=\psi(x)x'\}$$
denoted by $[X]$ in \cite{[BJK2005a]} -- see Lemma 2.2 and Propositions 2.3 and 2.4 there; the same paper in fact suggests a couple of other examples of closely related algebraic categories, where the colimit above could be described. Note also:
\begin{itemize}
	\item Since $[X]$ above is a $K$-subalgebra of $\mathrm{End}_K(X)\times\mathrm{End}_K(X)^{\mathrm{op}}$ such that, for every action $(B,X,\lambda,\rho)$ (as defined in Section 2), it contains the image of $\langle\lambda,\rho\rangle:B\to\mathrm{End}_K(X)\times\mathrm{End}_K(X)^{\mathrm{op}}$, one could rewrite our Section 2 using $[X]$ instead of $\mathrm{End}_K(X)\times\mathrm{End}_K(X)^{\mathrm{op}}$.
	\item Unlike \cite{[BJK2005a]}, \cite{[BJK2005]} used the symbol $[X]$ only for representing objects in action representable cases.  
\end{itemize}

\textbf{6.5.} There are important Bourn protomodular non-semi-abelian categories where split extensions still correspond to actions, possibly different from the actions we mentioned in Section 4. The first such example to consider would be the category of topological groups, which has representable `ordinary' actions, as shown by F. Cagliari and M. M. Clementino \cite{[CC2019]}; but there are also several others, including, say, the very recent one of bornological groups, due to F. Borceux and M. M. Clementino \cite{[BC2021]}. In all action representable examples, semi-abelian or not, the representing object can of course be obtained as the above-mentioned colimit $\mathrm{colim}(\mathsf{SplExt}(-,X)\to\mathscr{C})$, but, as far as I know, this colimit was never studied before as such. Would this give new action representability results? A closely related question would be: would using this colimit in connection with the results of J. R. A. Gray \cite{[G2017]} be helpful?          

{}

\end{document}